\documentclass{article}
\usepackage[final]{graphicx}
\usepackage{psfrag}
\usepackage{amsmath,amsfonts,amssymb,amsxtra,subeqnarray}

\newcommand {\Real}{\ensuremath{{\mathbb{R}}}}
\newcommand {\Natural}{\ensuremath{{\mathbb{N}}}}
\newcommand {\Complex}{\ensuremath{{\mathbb{C}}}}

\newcommand{\R}{\ensuremath{\mathcal R}}

\newcommand{\setS}{\ensuremath{\mathcal S}}

\newcommand{\U}{\ensuremath{\mathcal U}}

\newcommand{\X}{\ensuremath{\mathcal X}}
\newcommand{\Y}{\ensuremath{\mathcal Y}}

\newcommand{\N}{\ensuremath{\mathcal N}}

\newcommand{\yu}{\ensuremath{{\mathbf{u}}}}

\newcommand{\nal}{{\mbox{\rm null}}}

\newcommand{\xhat}{\hat x}

\newtheorem{theorem}{Theorem}
\newtheorem{algorithm}{Algorithm}
\newtheorem{corollary}{Corollary}

\newtheorem{lemma}{Lemma}

\newtheorem{definition}{Definition}
\newtheorem{remark}{Remark}
\newtheorem{assumption}{Assumption}

\newtheorem{proposition}{Proposition}
\newenvironment{proof}{\noindent {\bf Proof.}}{\hfill \hspace*{1pt}\hfill$\blacksquare$}

\begin{document}

\title{Deadbeat observer: construction via sets}
\author{S. Emre Tuna\footnote{Author is with Department of
Electrical and Electronics Engineering, Middle East Technical
University, 06800 Ankara, Turkey. Email: {\tt
tuna@eee.metu.edu.tr}}} \maketitle

\begin{abstract}
A geometric generalization of discrete-time linear deadbeat
observer is presented. The proposed method to generate a deadbeat
observer for a given nonlinear system is constructive and makes
use of sets that can be computed iteratively. For demonstration,
derivations of observer dynamics are provided for various example
systems. Based on the method, a simple algorithm that computes the
deadbeat gain for a linear system with scalar output is given.
\end{abstract}

\section{Introduction}

Observer design for linear systems is generally acknowledged to be
understood well enough. For discrete-time linear system $x^{+}=Ax$
with output $y=Cx$, Luenberger observer \cite{luenberger64}
dynamics read
\begin{eqnarray}\label{eqn:luenberger}
\xhat^{+}=A\xhat+L(y-C\xhat)
\end{eqnarray}
and designing the observer is nothing but choosing an observer
gain $L$ that places the eigenvalues of matrix $A-LC$ within the
unit circle. Simple and elegant, anyone would hardly doubt that
this construction is {\em the} construction for linear systems.
However, perhaps due arguably to over-elegance of the notation, it
is nontrivial to unearth the true mechanism (if it exists) running
behind Luenberger observer in order to generalize it in some {\em
natural} way for nonlinear systems. In this paper we aim to
provide a geometric interpretation of the righthand side of
\eqref{eqn:luenberger} for the particular case where matrix $A-LC$
is nilpotent, i.e., when the observer is deadbeat. Our
interpretation allows one to construct deadbeat observers for
nonlinear systems provided that certain conditions
(Assumption~\ref{assume:singleton} and
Assumption~\ref{assume:invariance}) hold. We now note and later
demonstrate that when the system is linear those assumptions are
minimal for a deadbeat observer to exist. The literature on
observers accommodates significant results. See, for instance,
\cite{karafyllis07,glad83,moraal95,valcher99,shamma99,fuhrmann06,besancon00,wong04}.

The toy example that we keep in the back of our mind while we
attempt to reach a generalization is the simple case where $A$ is
a rotation matrix in $\Real^{2}$
\begin{eqnarray*}
A =\left[\!\!\begin{array}{rr} \cos\theta &-\sin\theta\\
\sin\theta &\cos\theta
\end{array}\!\!\right]
\end{eqnarray*}
with angle of rotation $\theta$ different from $0$ and $\pi$.
Letting $y=x_{2}$, i.e., $C=[0\ \ 1]$, the deadbeat observer turns
out to be
\begin{eqnarray*}
\xhat^{+}=A\xhat+\left[\!\!\begin{array}{c} \cos2\theta/\sin\theta\\
\sin2\theta/\sin\theta
\end{array}\!\!\right](y-C\xhat)
\end{eqnarray*}
which can be rewritten as
\begin{eqnarray*}
\xhat^{+}=A\left(\xhat+\left[\!\!\begin{array}{c} \cot\theta\\
1
\end{array}\!\!\right](y-C\xhat)\right)
\end{eqnarray*}
Now we state the key observation in this paper: The term in
brackets is the intersection of two equivalence classes (sometimes
called congruence classes \cite{lax96}). Namely,
\begin{eqnarray*}
\xhat+\left[\!\!\begin{array}{c} \cot\theta\\
1
\end{array}\!\!\right](y-C\xhat)=(\xhat+A\,
\nal(C))\cap(x+\nal(C))
\end{eqnarray*}
as shown in Fig.~\ref{fig:intersect}.
\begin{figure}[h]
\begin{center}
\includegraphics[scale=0.6]{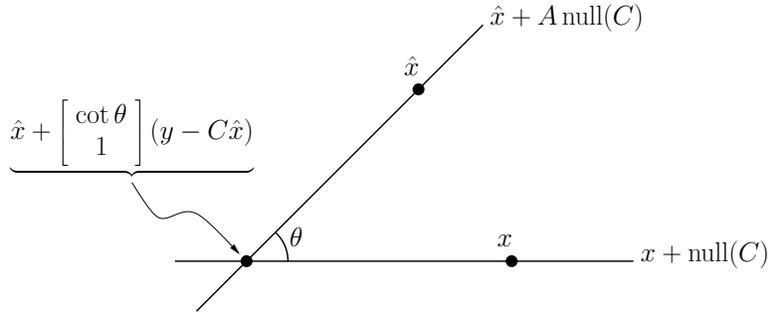}
\caption{Intersection of two equivalence
classes.}\label{fig:intersect}
\end{center}
\end{figure}
Based on this observation, one contribution of this paper is
intended to be in showing that such equivalence classes can be
defined even for nonlinear systems of arbitrary order, which in
turn allows one to construct deadbeat observers. There is another
possible contribution that is of more practical nature: We present
a simple algorithm that computes, for linear systems with scalar
output, deadbeat gain $L$ by iteratively intersecting linear
subspaces. (Devising reliable numerical techniques to compute
deadbeat gain for discrete-time linear systems had once been an
active field of research; see, for instance,
\cite{franklin82,lewis82,sugimoto93}.)

The remainder of the paper is organized as follows. Next section
contains some preliminary material. In Section~\ref{sec:def} we
give the formal problem definition. Section~\ref{sec:sets} is
where we describe the sets that we use in construction of the
deadbeat observer. We state and prove the main result in
Section~\ref{sec:main}. An extension of the main result where we
consider the case with input ($x^{+}=f(x,\,u)$) is in
Section~\ref{sec:input}. We provide examples in
Section~\ref{sec:ex}, where we construct deadbeat observers for
two different third order systems. In Section~\ref{sec:alg} we
present an algorithm to compute the deadbeat observer gain for a
linear system with scalar output.

\section{Preliminaries}

Identity matrix is denoted by $I$. Null space and range space of a
matrix $M\in\Real^{m\times n}$ are denoted by $\N(M)$ and $\R(M)$,
respectively. Given map $\mu:\X\to\Y$, $\mu^{-1}(\cdot)$ denotes
the {\em inverse} map in the general sense that for $y\in\Y$,
$\mu^{-1}(y)$ is the set of all $x\in\X$ satisfying $\mu(x)=y$.
That is, we will not need $\mu$ be bijective when talking about
its inverse. Note that $y\notin\mu(\X)$ will imply
$\mu^{-1}(y)=\emptyset$. Linear maps $x\mapsto Mx$ will not be
exempt from this notation. The reader should not think that $M$ is
a nonsingular matrix when we write $M^{-1}y$. (In our case $M$
need even not be square.) For instance, for $M=[0\ \ 0]$ we have
$M^{-1}y=\emptyset$ for $y\neq 0$ and $M^{-1}0=\Real^{2}$. The set
of nonnegative integers is denoted by $\Natural$ and $\Real_{>0}$
denotes the set of strictly positive real numbers.

\section{Problem definition}\label{sec:def}

Consider the following discrete-time system
\begin{subeqnarray}\label{eqn:system}
x^{+}&=&f(x)\\
y&=&h(x)
\end{subeqnarray}
where $x\in\X\subset\Real^{n}$ is the {\em state}, $x^{+}$ is the
state at the next time instant, and $y\in\Y\subset\Real^{m}$ is
the {\em output} or the {\em measurement}. The {\em solution} of
system~\eqref{eqn:system} at time $k\in\Natural$, starting at
initial condition $x\in\X$ is denoted by $\phi(k,\,x)$. Note that
$\phi(0,\,x)=x$ and $\phi(k+1,\,x)=f(\phi(k,\,x))$ for all $x$ and
$k$.

Now consider the following cascade system
\begin{subeqnarray}
x^{+}&=&f(x)\\
\xhat^{+}&\in&g(\xhat,\,h(x))\label{eqn:cascade}
\end{subeqnarray}
We denote {\em a} solution of subsystem~(\ref{eqn:cascade}b) by
$\psi(k,\,\xhat,\,x)$. We then have $\psi(0,\,\xhat,\,x)=\xhat$
and $\psi(k+1,\,\xhat,\,x)\in
g(\psi(k,\,\xhat,\,x),\,h(\phi(k,\,x)))$ for all $x$, $\xhat$, and
$k$. We now use \eqref{eqn:cascade} to define deadbeat observer.

\begin{definition}
Given $g:\X\times\Y\rightrightarrows\X$, system
\begin{eqnarray*}
\xhat^{+}\in g(\xhat,\,y)
\end{eqnarray*}
is said to be a {\em deadbeat observer for
system~\eqref{eqn:system}} if there exists $p\geq 1$ such that
{\em all} solutions of system~\eqref{eqn:cascade} satisfy
\begin{eqnarray*}
\psi(k,\,\xhat,\,x)=\phi(k,\,x)
\end{eqnarray*}
for all $x,\,\xhat\in\X$ and $k\geq p$.
\end{definition}

\begin{definition}
System~\eqref{eqn:system} is said to be {\em deadbeat observable}
if there exists a deadbeat observer for it.
\end{definition}

In this paper we present a procedure to construct a deadbeat
observer for system~\eqref{eqn:system} provided that certain
conditions (Assumption~\ref{assume:singleton} and
Assumption~\ref{assume:invariance}) hold. Our construction will
make use of some sets, which we define in the next section. Before
moving on into the next section, however, we choose to remind the
reader of a standard fact regarding the observability of linear
systems. Then we provide a Lemma~\ref{lem:subspace} as a geometric
equivalent of that well-known result. Lemma~\ref{lem:subspace}
will find use later when we attempt to interpret and display the
generality of the assumptions we will have made.

The following criterion, known as Popov-Belevitch-Hautus (PBH)
test, is an elegant tool for checking (deadbeat) observability.

\begin{proposition}[PBH test]
The linear system
\begin{subeqnarray}\label{eqn:linsystem}
x^{+}&=&Ax\\
y&=&Cx
\end{subeqnarray}
with $A\in\Real^{n\times n}$ and $C\in\Real^{m\times n}$ is
deadbeat observable if and only if
\begin{eqnarray}\label{eqn:PBH}
{\rm rank}\left[\!\!\begin{array}{cc}A-\lambda I\\
C\end{array}\!\!\right]=n\quad \mbox{for all}\quad \lambda\neq 0
\end{eqnarray}
where $\lambda$ is a complex scalar.
\end{proposition}

The below result is a geometric equivalent of PBH test.

\begin{lemma}\label{lem:subspace}
Given $A\in\Real^{n\times n}$ and $C\in\Real^{m\times n}$, let
subspace $\setS_{k}$ of $\Real^{n}$ be defined as
$\setS_{k}:=A\setS_{k-1}\cap\setS_{0}$ for $k=1,\,2,\,\ldots$ with
$\setS_{0}:=\N(C)$. Then system~\eqref{eqn:linsystem} is deadbeat
observable if and only if
\begin{eqnarray}\label{eqn:SET}
\setS_{n}=\{0\}\,.
\end{eqnarray}
\end{lemma}

\begin{proof}
For simplicity we provide the demonstration for the case where
each $\setS_{k}$ is a subspace of $\Complex^{n}$ (over field
$\Complex$). The case $\setS_{k}\subset\Real^{n}$ is a little
longer to prove yet it is true.

We first show \eqref{eqn:SET}$\implies$\eqref{eqn:PBH}. Suppose
\eqref{eqn:PBH} fails. That is, there exists an eigenvector
$v\in\Complex^{n}$ and a nonzero eigenvalue $\lambda\in\Complex$
such that $Av=\lambda v$ and $Cv=0$. Now suppose for some $k$ we
have $v\in\setS_{k}$. Then, since $v$ is an eigenvector with a
nonzero eigenvalue, we can write $v\in A\setS_{k}$. Observe that
$v\in\setS_{0}$ for $Cv=0$. As a result $v\in
A\setS_{k}\cap\setS_{0}=\setS_{k+1}$. By induction therefore we
have $v\in\setS_{k}$ for all $k$, which means that \eqref{eqn:SET}
fails.

Now we demonstrate the other direction
\eqref{eqn:PBH}$\implies$\eqref{eqn:SET}. We first claim that
$\setS_{k+1}\subset\setS_{k}$ for all $k$. We use induction to
justify our claim. Suppose $\setS_{k+1}\subset\setS_{k}$ for some
$k$. Then we can write
\begin{eqnarray*}
\setS_{k+2}
&=&A\setS_{k+1}\cap\setS_{0}\\
&\subset&A\setS_{k}\cap\setS_{0}\\
&=&\setS_{k+1}\,.
\end{eqnarray*}
Since $\setS_{1}\subset\setS_{0}$ our claim is valid. A trivial
implication of our claim then follows:
$\dim\setS_{k+1}\leq\dim\setS_{k}$ for all $k$. Let us now suppose
\eqref{eqn:SET} fails. That is, $\dim\setS_{n}\geq 1$. Note that
$\dim\setS_{0}\leq n$. Therefore $\dim\setS_{n}\geq 1$ and
$\dim\setS_{k+1}\leq\dim\setS_{k}$ imply the existence of some
$\ell\in\{0,\,1,\,\ldots,\,n-1\}$ such that
$\dim\setS_{\ell+1}=\dim\setS_{\ell}\geq 1$. Since
$\setS_{\ell+1}\subset\setS_{\ell}$, both $\setS_{\ell+1}$ and
$\setS_{\ell}$ having the same dimension implies
$\setS_{\ell+1}=\setS_{\ell}$. Hence we obtained
$\setS_{\ell}=A\setS_{\ell}\cap\setS_{0}$ which allows us to write
$\setS_{\ell}\subset A\setS_{\ell}$. Since
$\dim\setS_{\ell}\geq\dim A\setS_{\ell}$ we deduce that
$\setS_{\ell}=A\setS_{\ell}$. Since $\dim\setS_{\ell}\geq 1$,
equality $A\setS_{\ell}=\setS_{\ell}$ implies that there exists an
eigenvector $v\in\setS_{\ell}$ and a nonzero eigenvalue
$\lambda\in\Complex$ such that $Av=\lambda v$. Note also that
$Cv=0$ because $\setS_{\ell}\subset\setS_{0}$. Hence
\eqref{eqn:PBH} fails.
\end{proof}

\begin{remark}\label{rem:dimension}
It is clear from the proof that if \eqref{eqn:SET} fails then
$\dim\setS_{k}\geq 1$ for all $k$.
\end{remark}

\section{Sets}\label{sec:sets}

In this section we define certain sets (more formally, {\em
equivalence classes}) associated with system~\eqref{eqn:system}.
For $x\in\X$ we define
\begin{eqnarray*}
[x]_{0}:=h^{-1}(h(x))\,.
\end{eqnarray*}
Note that when $h(x)=Cx$, where $C\in\Real^{m\times n}$, we have
$[x]_{0}=x+\N(C)$. We then let for $k=0,\,1,\,\ldots$
\begin{eqnarray*}
[x]_{k+1}:=[x]_{k}^{+}\cap[x]_{0}
\end{eqnarray*}
where
\begin{eqnarray*}
[x]_{k}^{+}:=f([f^{-1}(x)]_{k})\,.
\end{eqnarray*}
Note that $[x]_{k}^{+}=\emptyset$ when $x\notin f(\X)$ since then
$f^{-1}(x)=\emptyset$.

\begin{remark}\label{rem:subset}
Note that $[x]_{k+1}\subset[x]_{k}$ and
$[x]^{+}_{k+1}\subset[x]^{+}_{k}$ for all $x$ and $k$.
\end{remark}

The following two assumptions will be invoked in our main theorem.
In hope of making them appear somewhat meaningful and revealing
their generality we provide the conditions that they would boil
down to for linear systems.

\begin{assumption}\label{assume:singleton}
There exists $p\geq 1$ such that, for each $x\in\X$, set
$[x]_{p-1}$ is either singleton or empty set.
\end{assumption}

Assumption~\ref{assume:singleton} is equivalent to deadbeat
observability for linear systems. Below result formalizes this.

\begin{theorem}
Linear system~\eqref{eqn:linsystem} is deadbeat observable if and
only if Assumption~\ref{assume:singleton} holds.
\end{theorem}

\begin{proof}
Let $\setS_{k}$ for $k=0,\,1,\,\ldots$ be defined as in
Lemma~\ref{lem:subspace}. Note then that $[x]_{0}=x+\setS_{0}$. We
claim that the following holds
\begin{eqnarray}\label{eqn:equivalence}
[x]_{k}=\left\{\begin{array}{cl} x+\setS_{k}&\quad\mbox{for}\quad
x\in\R(A^{k})\\
\emptyset&\quad\mbox{for}\quad x\notin\R(A^{k})
\end{array}\right.
\end{eqnarray}
for all $k$. We employ induction to establish our claim. Suppose
\eqref{eqn:equivalence} holds for some $k$. Then we can write
\begin{eqnarray*}
[x]_{k}^{+}
&=&A[A^{-1}x]_{k}\\
&=&A[A^{-1}x\cap\R(A^{k})]_{k}\,.
\end{eqnarray*}
Note that $A^{-1}x\cap\R(A^{k})\neq\emptyset$ if and only if
$x\in\R(A^{k+1})$. Since $[x]_{k+1}=[x]_{k}^{+}\cap[x]_{0}$, we
deduce that $[x]_{k+1}=\emptyset$ for $x\notin\R(A^{k+1})$.
Otherwise if $x\in\R(A^{k+1})$ then there exists some
$\eta\in\R(A^{k})$ such that $A\eta=x$. Using this $\eta$ we can
construct the equality $A^{-1}x=\eta+\N(A)$ and we can write
\begin{eqnarray*}
[x]_{k+1}
&=&[x]_{k}^{+}\cap[x]_{0}\\
&=&A[A^{-1}x]_{k}\cap[x]_{0}\\
&=&A[\eta+\N(A)]_{k}\cap[x]_{0}\\
&=&A(\eta+(\N(A)\cap\R(A^{k}))+\setS_{k})\cap[x]_{0}\\
&=&(A\eta+A\setS_{k})\cap(x+S_{0})\\
&=&(x+A\setS_{k})\cap(x+\setS_{0})\\
&=&x+(A\setS_{k}\cap\setS_{0})\\
&=&x+\setS_{k+1}\,.
\end{eqnarray*}
Since \eqref{eqn:equivalence} holds for $k=0$, our claim is valid.

Now suppose that the system is deadbeat observable. Then by
\eqref{eqn:equivalence} we see that
Assumption~\ref{assume:singleton} holds with $p=n+1$ thanks to
Lemma~\ref{lem:subspace}. If however the system is not deadbeat
observable, then by Remark~\ref{rem:dimension} $\dim \setS_{k}\geq
1$ for all $k$. We deduce by \eqref{eqn:equivalence} therefore
that $[0]_{k}$ can never be singleton nor is it empty. Hence
Assumption~\ref{assume:singleton} must fail.
\end{proof}

\begin{assumption}\label{assume:invariance}
Given $x,\,\xhat\in\X$ and $k$; $\xhat\in[x]_{k}^{+}$ implies
$[\xhat]_{k}^{+}=[x]_{k}^{+}$.
\end{assumption}

\begin{theorem}
Assumption~\ref{assume:invariance} comes for free for linear
system~\eqref{eqn:linsystem}.
\end{theorem}

\begin{proof}
Evident.
\end{proof}
\\

Last we let $[x]_{-1}^{+}:=\X$ and define map
$\pi:\X\times\Y\to\{-1,\,0,\,1,\,\ldots,\,p-2\}$ as
\begin{eqnarray*}
\pi(\xhat,\,y) := \max\,\{-1,\,0,\,1,\,\ldots,\,p-2\}\quad
\mbox{subject to}\quad [\xhat]^{+}_{\pi(\xhat,\,y)}\cap h^{-1}(y)
\neq\emptyset
\end{eqnarray*}
where $p$ is as in Assumption~\ref{assume:singleton}.

\section{The result}\label{sec:main}

Below is our main theorem.

\begin{theorem}\label{thm:main}
Suppose Assumptions~\ref{assume:singleton}-\ref{assume:invariance}
hold. Then system
\begin{eqnarray}\label{eqn:deadbeat}
\hat{x}^{+}\in f([\xhat]^{+}_{\pi(\xhat,\,y)}\cap h^{-1}(y))
\end{eqnarray}
is a deadbeat observer for system~\eqref{eqn:system}.
\end{theorem}

\begin{proof}
We claim the following
\begin{eqnarray}\label{eqn:claim}
\xhat\in [x]_{\ell-1}^{+}\implies \xhat^{+}\in[f(x)]_{\ell}^{+}
\end{eqnarray}
for all $\ell\in\{0,\,1,\,\ldots,\,p-1\}$. Let us prove our claim.
Note that $\xhat\in[x]_{\ell-1}^{+}$ yields
$[\xhat]_{\ell-1}^{+}=[x]_{\ell-1}^{+}$ by
Assumption~\ref{assume:invariance}. Since
$[x]_{\ell-1}^{+}\neq\emptyset$ we have
$[x]_{\ell-1}^{+}\cap[x]_{0}\neq\emptyset$ and, consequently,
$[\xhat]_{\ell-1}^{+}\cap[x]_{0}\neq\emptyset$.
Remark~\ref{rem:subset} then yields
$[\xhat]_{\pi(\xhat,\,h(x))}^{+}\subset[\xhat]_{\ell-1}^{+}$.
Starting from \eqref{eqn:deadbeat} we can proceed as
\begin{eqnarray}\label{eqn:yesnumber}
\hat{x}^{+}
&\in& f([\xhat]^{+}_{\pi(\xhat,\,y)}\cap h^{-1}(y))\nonumber\\
&=& f([\xhat]^{+}_{\pi(\xhat,\,h(x))}\cap h^{-1}(h(x)))\nonumber\\
&\subset& f([\xhat]^{+}_{\ell-1}\cap [x]_{0})\nonumber\\
&=& f([x]^{+}_{\ell-1}\cap [x]_{0})\nonumber\\
&=& f([x]_{\ell})\\
&\subset& f([f^{-1}(f(x))]_{\ell})\nonumber\\
&=& [f(x)]^{+}_{\ell}\,.\nonumber
\end{eqnarray}
Hence \eqref{eqn:claim} holds. In particular,
\eqref{eqn:yesnumber} gives us
\begin{eqnarray}\label{eqn:genco}
\xhat\in [x]_{\ell-1}^{+}\implies \xhat^{+}\in f([x]_{\ell})
\end{eqnarray}
for all $\ell\in\{0,\,1,\,\ldots,\,p-1\}$. Note that $\xhat\in
[x]^{+}_{-1}$ holds for all $x,\,\xhat$. Therefore
\eqref{eqn:claim} and Remark~\ref{rem:subset} imply the existence
of $\ell^{\ast}\in\{0,\,1,\,\ldots,\,p-1\}$ such that
\begin{eqnarray}\label{eqn:dede}
\psi(k,\,\xhat,\,x)\in[\phi(k,\,x)]^{+}_{p-2}
\end{eqnarray}
for all $k\geq\ell^{\ast}$. Also,
Assumption~\ref{assume:singleton} yields us
\begin{eqnarray}\label{eqn:ibo}
[\phi(k,\,x)]_{p-1}=\phi(k,\,x)
\end{eqnarray}
for all $k\geq p-1$. Combining \eqref{eqn:genco},
\eqref{eqn:dede}, and \eqref{eqn:ibo} we can write
\begin{eqnarray*}
\psi(k,\,\xhat,\,x)=\phi(k,\,x)
\end{eqnarray*}
for all $k\geq p$. Hence the result.
\end{proof}

\begin{corollary}\label{cor:foralg}
Consider linear system~\eqref{eqn:linsystem} with
$A\in\Real^{n\times n}$ and $C\in\Real^{1\times n}$. Suppose pair
$(C,\,A)$ is observable\footnote{That is, $\mbox{rank}\ [C^{T}\
A^{T}C^{T}\ \ldots\ A^{(n-1)T}C^{T}]=n$.}. Let $\setS_{k}$ for
$k=0,\,1,\,\ldots$ be defined as in Lemma~\ref{lem:subspace}. Then
system
\begin{eqnarray*}
\xhat^{+}=A((\xhat+A\setS_{n-2})\cap(x+\setS_{0}))
\end{eqnarray*}
is a deadbeat observer for system~\eqref{eqn:linsystem}.
\end{corollary}

\section{System with input}\label{sec:input}

In this section we look at the case where the evolution of system
to be observed is dependent not only on the initial condition but
also on some exogenous signal, which we call the input. To
construct a deadbeat observer for such system we again make use of
sets.

Consider the system
\begin{subeqnarray}\label{eqn:systemwu}
x^{+}&=&f(x,\,u)\\
y&=&h(x)
\end{subeqnarray}
where $u\in\U\subset\Real^{q}$ is the {\em input} or some known
{\em disturbance} (e.g. time). Let ${\bf
u}=(u_{0},\,u_{1},\,\ldots)$, $u_{k}\in\U$, denote an input
sequence. The {\em solution} of system~\eqref{eqn:systemwu} at
time $k$, starting at initial condition $x\in\X$, and having
evolved under the influence of input sequence $\yu$ is denoted by
$\phi(k,\,x,\,\yu)$. Note that $\phi(0,\,x,\,\yu)=x$ and
$\phi(k+1,\,x,\,\yu)=f(\phi(k,\,x,\,\yu),\,u_{k})$ for all $x$,
$\yu$, and $k$.

Now consider the following cascade system
\begin{subeqnarray}
x^{+}&=&f(x,\,u)\\
\xhat^{+}&\in&g(\xhat,\,h(x),\,u)\label{eqn:cascadewu}
\end{subeqnarray}
We denote a solution of subsystem~(\ref{eqn:cascadewu}b) by
$\psi(k,\,\xhat,\,x,\,\yu)$. We then have
$\psi(0,\,\xhat,\,x,\,\yu)=\xhat$ and
$\psi(k+1,\,\xhat,\,x,\,\yu)\in
g(\psi(k,\,\xhat,\,x,\,\yu),\,h(\phi(k,\,x,\,\yu)),\,u_{k})$ for
all $x$, $\xhat$, $\yu$, and $k$.

\begin{definition}
Given $g:\X\times\Y\times\U\rightrightarrows\X$, system
\begin{eqnarray*}
\xhat^{+}\in g(\xhat,\,y,\,u)
\end{eqnarray*}
is said to be a {\em deadbeat observer for
system~\eqref{eqn:systemwu}} if there exists $p\geq 1$ such that
solutions of system~\eqref{eqn:cascadewu} satisfy
\begin{eqnarray*}
\psi(k,\,\xhat,\,x,\,\yu)=\phi(k,\,x,\,\yu)
\end{eqnarray*}
for all $x$, $\xhat$, $\yu$, and $k\geq p$.
\end{definition}

How to define sets $[x]_{k}$ and $[x]_{k}^{+}$ for
system~\eqref{eqn:systemwu} is obvious. We again let
\begin{eqnarray*}
[x]_{0}:=h^{-1}(h(x))\,.
\end{eqnarray*}
and (for $k=0,\,1,\,\ldots$)
\begin{eqnarray*}
[x]_{k+1}:=[x]_{k}^{+}\cap[x]_{0}
\end{eqnarray*}
this time with
\begin{eqnarray*}
[x]_{k}^{+}:=\bigcup_{f(\eta,\,u)=x} f([\eta]_{k},\,u)\,.
\end{eqnarray*}
The following result is a generalization of
Theorem~\ref{thm:main}. (The demonstration is parallel to that of
Theorem~\ref{thm:main} and hence omitted.)
\begin{theorem}\label{thm:mainu}
Suppose Assumptions~\ref{assume:singleton}-\ref{assume:invariance}
hold. Then system
\begin{eqnarray*}
\hat{x}^{+}\in f([\xhat]^{+}_{\pi(\xhat,\,y)}\cap h^{-1}(y),\,u)
\end{eqnarray*}
is a deadbeat observer for system~\eqref{eqn:systemwu}.
\end{theorem}

\section{Examples}\label{sec:ex}

Here, for two third order nonlinear systems, we construct deadbeat
observers. In the first example we study a simple autonomous
homogeneous system and show that the construction yields a
homogeneous observer. Hence our method may be thought to be
somewhat {\em natural} in the vague sense that the observer it
generates inherits certain intrinsic properties of the system. In
the second example we aim to provide a demonstration on observer
construction for a system with input.

\subsection{Homogeneous system}

Consider system~\eqref{eqn:system} with
\begin{eqnarray*}
f(x):=\left[\!\!
\begin{array}{c}
x_{2}\\
x_{3}^{1/3}\\
x_{1}^{3}+x_{2}^{3}
\end{array}\!\!\right]\quad\mbox{and}\quad h(x):=x_{1}
\end{eqnarray*}
where $x=[x_{1}\ x_{2}\ x_{3}]^{T}$. Let $\X=\Real^{3}$ and
$\Y=\Real$. If we let dilation $\Delta_{\lambda}$ be
\begin{eqnarray*}
\Delta_{\lambda}:=\left[\!\!
\begin{array}{ccc}
\lambda&0&0\\
0&\lambda&0\\
0&0&\lambda^{3}
\end{array}\!\!\right]
\end{eqnarray*}
with $\lambda\in\Real$, then we realize that
\begin{eqnarray*}
f(\Delta_{\lambda}x)=\Delta_{\lambda}f(x)\quad\mbox{and}\quad
h(\Delta_{\lambda}x)=\lambda h(x)\,.
\end{eqnarray*}
That is, the system is homogeneous \cite{rinehart09} with respect
to dilation $\Delta$. Before describing the relevant sets
$[x]_{k}$ and $[x]_{k}^{+}$ we want to mention that $f$ is
bijective and its inverse is
\begin{eqnarray}\label{eqn:finv}
f^{-1}(x)=\left[\!\!
\begin{array}{c}
(x_{3}-x_{1}^{3})^{1/3}\\
x_{1}\\
x_{2}^{3}
\end{array}\!\!\right]
\end{eqnarray}
Since $h(x)=x_{1}$ we can write
\begin{eqnarray}\label{eqn:x0}
[x]_{0}=\left\{ \left[\!\!\begin{array}{c} x_{1}\\
\alpha\\
\beta
\end{array}\!\!\right]:\alpha,\,\beta\in\Real
\right\}
\end{eqnarray}
By \eqref{eqn:finv} we can then proceed as
\begin{eqnarray}\label{eqn:x0plus}
[x]_{0}^{+}&=&f([f^{-1}(x)]_{0})\nonumber\\
&=&f\left(\left\{ \left[\!\!\begin{array}{c} (x_{3}-x_{1}^{3})^{1/3}\\
\gamma\\
\delta
\end{array}\!\!\right]:\gamma,\,\delta\in\Real
\right\}\right)\nonumber\\
&=&\left\{ f\left(\left[\!\!\begin{array}{c} (x_{3}-x_{1}^{3})^{1/3}\\
\gamma\\
\delta
\end{array}\!\!\right]\right):\gamma,\,\delta\in\Real
\right\}\nonumber\\
&=&\left\{ \left[\!\!\begin{array}{c} \gamma\\
\delta^{1/3}\\
x_{3}-x_{1}^{3}+\gamma^{3}
\end{array}\!\!\right]:\gamma,\,\delta\in\Real
\right\}
\end{eqnarray}
Recall that $[x]_{1}=[x]_{0}^{+}\cap[x]_{0}$. Therefore
intersecting sets \eqref{eqn:x0} and \eqref{eqn:x0plus} we obtain
\begin{eqnarray*}
[x]_{1}=\left\{ \left[\!\!\begin{array}{c} x_{1}\\
\alpha\\
x_{3}
\end{array}\!\!\right]:\alpha\in\Real
\right\}
\end{eqnarray*}
We can now construct $[x]_{1}^{+}$ as
\begin{eqnarray}\label{eqn:x1plus}
[x]_{1}^{+}&=&f([f^{-1}(x)]_{1})\nonumber\\
&=&f\left(\left\{ \left[\!\!\begin{array}{c} (x_{3}-x_{1}^{3})^{1/3}\\
\gamma\\
x_{2}^{3}
\end{array}\!\!\right]:\gamma\in\Real
\right\}\right)\nonumber\\
&=&\left\{ f\left(\left[\!\!\begin{array}{c} (x_{3}-x_{1}^{3})^{1/3}\\
\gamma\\
x_{2}^{3}
\end{array}\!\!\right]\right):\gamma\in\Real
\right\}\nonumber\\
&=&\left\{ \left[\!\!\begin{array}{c} \gamma\\
x_{2}\\
x_{3}-x_{1}^{3}+\gamma^{3}
\end{array}\!\!\right]:\gamma\in\Real
\right\}
\end{eqnarray}
Now note that sets \eqref{eqn:x0} and \eqref{eqn:x1plus} intersect
at a single point. In particular,
$[x]_{2}=[x]_{1}^{+}\cap[x]_{0}=x$. Therefore
Assumption~\ref{assume:singleton} is satisfied with $p=3$. Observe
also that
\begin{eqnarray*}
[\xhat]_{1}^{+}\cap h^{-1}(y) &=&
\left\{ \left[\!\!\begin{array}{c} \gamma\\
\xhat_{2}\\
\xhat_{3}-\xhat_{1}^{3}+\gamma^{3}
\end{array}\!\!\right]:\gamma\in\Real
\right\}\cap\left\{ \left[\!\!\begin{array}{c} y\\
\alpha\\
\beta
\end{array}\!\!\right]:\alpha,\,\beta\in\Real
\right\}\\
&=&\left[\!\!\begin{array}{c} y\\
\xhat_{2}\\
\xhat_{3}-\xhat_{1}^{3}+y^{3}
\end{array}\!\!\right]
\end{eqnarray*}
which means that $\pi(\xhat,\,y)=p-2=1$ for all $\xhat$ and $y$.
The dynamics of the deadbeat observer then read
\begin{eqnarray*}
\xhat^{+}&=&f([\xhat]_{1}^{+}\cap h^{-1}(y))\\
&=&\left[\!\!\begin{array}{c}
\xhat_{2}\\
(\xhat_{3}-\xhat_{1}^{3}+y^{3})^{1/3}\\
\xhat_{2}^{3}+y^{3}
\end{array}\!\!\right]
\end{eqnarray*}
We finally notice that
\begin{eqnarray*}
f([\Delta_{\lambda}\xhat]_{1}^{+}\cap h^{-1}(\lambda
y))=\Delta_{\lambda}f([\xhat]_{1}^{+}\cap h^{-1}(y))\,.
\end{eqnarray*}
That is, the deadbeat observer also is homogeneous with respect to
dilation $\Delta$.

\subsection{System with input}
Our second example is again a third order system, this time
however with an input. Consider system~\eqref{eqn:systemwu} with
\begin{eqnarray*}
f(x,\,u):=\left[\!\!
\begin{array}{c}
x_{1}x_{2}x_{3}\\
x_{3}/x_{1}\\
\sqrt{x_{1}x_{2}u}
\end{array}\!\!\right]
\quad\mbox{and}\quad h(x):=x_{1}\,.
\end{eqnarray*}
Let $\X=\Real_{>0}^{3}$, $\Y=\Real_{>0}$, and $\U=\Real_{>0}$. Let
us construct the relevant sets $[x]_{k}$ and $[x]_{k}^{+}$. We
begin with $[x]_{0}$.
\begin{eqnarray}\label{eqn:x0nd}
[x]_{0}=\left\{ \left[\!\!\begin{array}{c} x_{1}\\
\alpha\\
\beta
\end{array}\!\!\right]:\alpha,\,\beta>0
\right\}
\end{eqnarray}
Note that $f$ satisfies the following
\begin{eqnarray*}
f\left(\left[\!\!
\begin{array}{c}
x_{1}u/(x_{2}x_{3}^{2})\\
x_{2}x_{3}^{4}/(x_{1}u^{2})\\
x_{1}u/x_{3}^{2}
\end{array}\!\!\right],\,u\right)=x
\end{eqnarray*}
for all $x$ and $u$. Hence we can write
\begin{eqnarray}\label{eqn:x0plusnd}
[x]_{0}^{+}&=&\bigcup_{u\in\U}f\left(\left[\!\!
\begin{array}{c}
x_{1}u/(x_{2}x_{3}^{2})\\
x_{2}x_{3}^{4}/(x_{1}u^{2})\\
x_{1}u/x_{3}^{2}
\end{array}\!\!\right]_{0},\,u\right)\nonumber\\
&=&\bigcup_{u\in\U}f\left(\left\{ \left[\!\!\begin{array}{c} x_{1}u/(x_{2}x_{3}^{2})\\
\gamma\\
\delta
\end{array}\!\!\right]:\gamma,\,\delta>0
\right\},\,u\right)\nonumber\\
&=&\bigcup_{u\in\U}f\left(\left\{ \left[\!\!\begin{array}{c} x_{1}u/(x_{2}x_{3}^{2})\\
x_{2}x_{3}^{4}\gamma/(x_{1}u^{2})\\
x_{1}u\delta/x_{3}^{2}
\end{array}\!\!\right]:\gamma,\,\delta>0
\right\},\,u\right)\nonumber\\
&=&\bigcup_{u\in\U}\left\{ f\left(\left[\!\!\begin{array}{c} x_{1}u/(x_{2}x_{3}^{2})\\
x_{2}x_{3}^{4}\gamma/(x_{1}u^{2})\\
x_{1}u\delta/x_{3}^{2}
\end{array}\!\!\right],\,u\right):\gamma,\,\delta>0
\right\}\nonumber\\
&=&\left\{\left[\!\!\begin{array}{c} x_{1}\gamma\delta\\
x_{2}\delta\\
x_{3}\sqrt{\gamma}
\end{array}\!\!\right] :\gamma,\,\delta>0
\right\}
\end{eqnarray}
Since $[x]_{1}=[x]_{0}^{+}\cap[x]_{0}$, intersecting sets
\eqref{eqn:x0nd} and \eqref{eqn:x0plusnd} we obtain
\begin{eqnarray*}
[x]_{1}=\left\{ \left[\!\!\begin{array}{c} x_{1}\\
x_{2}/\alpha^{2}\\
x_{3}\alpha
\end{array}\!\!\right]:\alpha>0
\right\}
\end{eqnarray*}
We can now construct $[x]_{1}^{+}$ as
\begin{eqnarray}\label{eqn:x1plusnd}
[x]_{1}^{+}&=&\bigcup_{u\in\U}f\left(\left[\!\!
\begin{array}{c}
x_{1}u/(x_{2}x_{3}^{2})\\
x_{2}x_{3}^{4}/(x_{1}u^{2})\\
x_{1}u/x_{3}^{2}
\end{array}\!\!\right]_{1},\,u\right)\nonumber\\
&=&\bigcup_{u\in\U}f\left(\left\{ \left[\!\!\begin{array}{c} x_{1}u/(x_{2}x_{3}^{2})\\
x_{2}x_{3}^{4}/(x_{1}u^{2}\gamma^{2})\\
x_{1}u\gamma/x_{3}^{2}
\end{array}\!\!\right]:\gamma>0
\right\},\,u\right)\nonumber\\
&=&\left\{\left[\!\!\begin{array}{c} x_{1}/\gamma\\
x_{2}\gamma\\
x_{3}/\gamma
\end{array}\!\!\right] :\gamma>0
\right\}
\end{eqnarray}
Now note that sets \eqref{eqn:x0nd} and \eqref{eqn:x1plusnd}
intersect at a single point. In particular,
$[x]_{2}=[x]_{1}^{+}\cap[x]_{0}=x$. Therefore
Assumption~\ref{assume:singleton} is satisfied with $p=3$. Observe
also that
\begin{eqnarray*}
[\xhat]_{1}^{+}\cap h^{-1}(y) &=& \left\{\left[\!\!\begin{array}{c} \xhat_{1}/\gamma\\
\xhat_{2}\gamma\\
\xhat_{3}/\gamma
\end{array}\!\!\right] :\gamma>0
\right\}\cap \left\{ \left[\!\!\begin{array}{c} y\\
\alpha\\
\beta
\end{array}\!\!\right]:\alpha,\,\beta>0
\right\}\\ &=&
\left[\!\!\begin{array}{c} y\\
\xhat_{1}\xhat_{2}/y\\
\xhat_{3}y/\xhat_{1}
\end{array}\!\!\right]
\end{eqnarray*}
which means that $\pi(\xhat,\,y)=p-2=1$ for all $\xhat$ and $y$.
The dynamics of the deadbeat observer then read
\begin{eqnarray*}
\xhat^{+}&=&f([\xhat]_{1}^{+}\cap h^{-1}(y),\,u)\\
&=&\left[\!\!\begin{array}{c}
\xhat_{2}\xhat_{3}y\\
\xhat_{3}/\xhat_{1}\\
\sqrt{\xhat_{1}\xhat_{2}u}
\end{array}\!\!\right]
\end{eqnarray*}

\section{An algorithm for deadbeat gain}\label{sec:alg}

In this section we provide an algorithm to compute the deadbeat
observer gain for a linear system with scalar output. (The
algorithm directly follows from Corollary~\ref{cor:foralg}.)
Namely, given an observable pair $(C,\,A)$ with
$C\in\Real^{1\times n}$ and $A\in\Real^{n\times n}$, we provide a
procedure to compute the gain $L\in\Real^{n\times 1}$ that renders
matrix $A-LC$ nilpotent. Below we let $\nal(\cdot)$ be some
function such that, given matrix $M\in\Real^{m\times n}$ whose
dimension of null space is $k$, $\nal(M)$ is some $n\times k$
matrix whose columns span the null space of $M$.

\begin{algorithm}\label{alg:db}
Given $C\in\Real^{1\times n}$ and $A\in\Real^{n\times n}$, the
following algorithm generates deadbeat gain $L\in\Real^{n\times
1}$.
\begin{eqnarray*}
&& X = \nal(C)\\
&& \mbox{{\bf for}}\quad i = 1:n-2\\
&& \qquad X = \nal\left(\left[
\begin{array}{c}
C\\ \nal((AX)^{T})^{T}
\end{array}
\right]\right)\\
&&\mbox{{\bf end}}\\
&&L_{\rm pre}=AX\\
&&L = \frac{AL_{\rm pre}}{CL_{\rm pre}}
\end{eqnarray*}
\end{algorithm}

For the interested reader we below give a M{\small ATLAB} code.
Exploiting Algorithm~\ref{alg:db}, this code generates a function
(which we named {\tt dbLfun}) whose inputs are matrices $C$ and
$A$. The output of the function, as its name indicates, is the
deadbeat gain $L$.
\begin{verbatim}
function L = dbLfun(C,A)
X = null(C);
for i = 1:length(A)-2
    X = null([C;null((A*X)')']);
end
Lpre = A*X;
L = A*Lpre/(C*Lpre);
\end{verbatim}
One can also use the built-in M{\small ATLAB} function {\tt acker}
to compute the deadbeat gain. We can therefore compare {\tt
dbLfun} with {\tt acker} via a numerical experiment.
Table~\ref{tab:one} gives the experimental results. Number $n$ is
the dimension of the system (that is, the number of columns of $A$
matrix) and the numbers at the bottom row are the percentages of
the cases (among $10^4$ random trials for each $n$) in which {\tt
dbLfun} performed better than {\tt acker}. How we determine which
one is better in a given case is as follows. Given pair $(C,\,A)$,
we let $L_{1}$ be the gain resulting from {\tt dbLfun(C,A)} and
$L_{2}$ be the gain given by {\tt acker(A',C',zeros(n,1))'}. Then
we compare norms $|(A-L_{1}C)^n|$ and $|(A-L_{2}C)^n|$, neither of
which is zero due to round-off errors. The function yielding the
smaller norm is considered to be better.

\begin{table}\caption{Percentages of cases where {\tt dbLfun} performed better than {\tt acker}.}
\begin{center}\label{tab:one}
\begin{tabular}{|c|c|c|c|c|c|c|c|}
\hline
n=3 & n=4 & n=5 & n=6 & n=7 & n=8 & n=9 & n=10\\
\hline \%51 & \%60 & \%67 & \%74 & \%80 & \%85 & \%87 & \%91\\
\hline
\end{tabular}
\end{center}
\end{table}

\section{Conclusion}

For nonlinear systems a method to construct a deadbeat observer is
proposed. The resultant observer can be considered as a
generalization of the linear deadbeat observer. The construction
makes use of sets that are generated iteratively. Through such
iterations, observers are derived for two academic examples. Also,
for computing the deadbeat gain for a linear system with scalar
output, an algorithm that works no worse than an already existing
one is given.

\bibliographystyle{plain}
\bibliography{references}
\end{document}